\documentclass[12pt,reqno]{amsart}

\usepackage{amssymb,amsfonts,latexsym,amstext,epsfig,color,graphicx,tikz,graphics,float,fp,amsthm,cite,caption,mathtools,hyphenat,amsmath,systeme, lscape,rotating,subcaption,setspace,hhline}
\usepackage[left=2.5cm,top=2.5cm,right=2.5cm,bottom=2.5cm]{geometry}

\newtheorem{theorem}{Theorem}[section]
\newtheorem{lemma}[theorem]{Lemma}
\newtheorem{corollary}[theorem]{Corollary}
\theoremstyle{definition}

\newcommand{\floor}[1]{\left\lfloor #1 \right\rfloor}
\newcommand{\ceil}[1]{\left\lceil #1 \right\rceil}
\newcommand{\T}{\mathcal{T}}
\newcommand{\loc}[1]{\zeta(#1)}

\usetikzlibrary{decorations.pathreplacing}
\usetikzlibrary{decorations.pathmorphing}
\usetikzlibrary{positioning}
\usetikzlibrary{snakes}

\newcommand{\capt}{\mathrm{capt}}

\begin{document}
\title{The localization capture time of a graph}\thanks{The second, third, and fifth authors were supported by NSERC}
\author[N.C.\ Behague]{Natalie C.\ Behague}
\author[A.\ Bonato]{Anthony Bonato}
\author[M.A.\ Huggan]{Melissa A.\ Huggan}
\author[T.G.\ Marbach]{Trent G.\ Marbach}
\author[B.\ Pittman]{Brittany Pittman}
\address[A1, A2, A3, A4, A5]{Department of Mathematics, Ryerson University, Toronto, Canada,  M5B 2K3.}
\email[A1]{(A1) nbehague@ryerson.ca}
\email[A2]{(A2) abonato@ryerson.ca}
\email[A3]{(A3) melissa.huggan@ryerson.ca}
\email[A4]{(A4) trent.marbach@ryerson.ca}
\email[A5]{(A5) brittany.pittman@ryerson.ca}

\begin{abstract}
The localization game is a pursuit-evasion game analogous to Cops and Robbers, where the robber is invisible and the cops send distance probes in an attempt to identify the location of the robber. We present a novel graph parameter called the capture time, which measures how long the localization game lasts assuming optimal play. We conjecture that the capture time is linear in the order of the graph, and show that the conjecture holds for graph families such as trees and interval graphs. We study bounds on the capture time for trees and its monotone property on induced subgraphs of trees and more general graphs. We give upper bounds for the capture time on the incidence graphs of projective planes. We finish with new bounds on the localization number and capture time using treewidth.
\end{abstract}


\maketitle
\section{Introduction}

\emph{Pursuit-evasion games} are combinatorial models for the detection or neutralization of an adversary's activity on a graph. In such models, agents or cops are attempting to capture an adversary or robber loose on the vertices of a graph. The players move at alternating ticks of the clock, and have restrictions on their movements or relative speed depending on the game played. The most studied such game is Cops and Robbers, where the cops and robber can only move to vertices with which they share an edge. The cop number is the minimum number of cops needed to guarantee the robber's capture. How the players move and the rules of capture depend on which variant is studied. These variants are motivated by problems in practice or inspired by foundational issues in computer science, discrete mathematics, and artificial intelligence, such as robotics and network security. For surveys of pursuit-evasion games, see~\cite{bp,by}, and see \cite{BN} for more background on Cops and Robbers.

The localization game was first introduced for one cop by Seager~\cite{seager1,seager2} and was further studied in \cite{BHM,BHM1,BK,nisse1,BDELM,car,DEFMP,DFP,has}. In the localization game, two players play on a connected graph, with one player controlling a set of $k$ \emph{cops}, where $k$ is a positive integer, and the second controlling a single \emph{robber}. The robber is invisible to the cops during gameplay. When there is no ambiguity, we identify a player with the vertex it occupies. The game is played over a sequence of discrete time-steps; a \emph{round} is a cop move and a subsequent robber move.

The robber occupies a vertex of the graph, and when the robber is ready to move during a round, they may move to a neighboring vertex or remain on their current vertex.  A move for the cops is a placement of cops on a set of vertices. Note that the cops are not limited to moving to neighboring vertices. At the beginning of the game, the robber chooses a starting vertex. After this, the cops move first, followed by the robber; thereafter, they move in alternate turns. Observe that any subset of cops may move in a given round. In each round, the cops occupy a set of vertices $u_1, u_2, \ldots , u_k$ and each cop sends out a \emph{cop probe}, which gives their distance $d_i$, from $u_i$ to the robber, where $1\le i \le k$. Hence, in each round, the cops determine a \emph{distance vector} $(d_1, d_2, \ldots , d_k)$ of cop probes. The cops win if they have a strategy to determine, after finitely many rounds, the vertex the robber occupies, at which time we say that the cops {\em capture} the robber. The robber wins if they are never captured.

For a connected graph $G$, define the \emph{localization number} of $G$, written $\loc{G}$, to be the least integer $k$ for which $k$ cops have a winning strategy over any possible strategy of the robber; that is, we consider the worst case for the cops in that the robber is \emph{omniscient} and so knows the entire strategy of the cops. As $\loc{G}$ is at most $n-1,$ the parameter is well-defined. Note that $\loc{G} \le  \beta(G),$ where $\beta(G)$ is the metric dimension of $G.$

In \cite{nisse1}, it was shown that $\loc{G}$ is bounded above by the pathwidth of $G$ and that the localization number is unbounded even on planar graphs obtained by adding a universal vertex to a tree. They also proved that computing $\loc{G}$ is \textbf{NP}-hard for graphs with diameter $2$. Bonato and Kinnersley~\cite{BK} studied the localization number for graphs based on their degeneracy. In ~\cite{BK}, they resolved a conjecture from \cite{nisse1} relating $\loc{G}$ and the chromatic number; further, they proved that the localization number of outerplanar graphs is at most 2, and they proved an asymptotically tight upper bound on the localization number of the hypercube. The localization number of the incidence graphs of designs was studied in \cite{BHM}. In particular, they gave exact values for the localization number of the incidence graphs of projective and affine planes, and bounds for the incidence graphs of Steiner systems and transversal designs. The localization number of graph products was considered in \cite{cartesianproducts}, and in diameter 2 graphs such as Kneser, Moore, and polarity graphs in \cite{BHM1}. Localization was studied in random binomial graphs in \cite{DFP,DEFMP} and in random geometric graphs in \cite{MP}.

In the present paper, we focus not only on the number of cops needed to capture the robber in the localization game, but the \emph{time} or minimum number of rounds it takes to do so. For a graph $G$ and an integer $k \ge \zeta(G),$ the corresponding optimization parameter is $\mathrm{capt}_{\zeta,k}(G),$ which is the minimum number of rounds for the cops to capture the robber.  If $k = \zeta(G),$ then we simplify this to $\mathrm{capt}_{\zeta}(G)$. We refer to this graph parameter as the \emph{capture time} of $G$ for the localization game. Note that we assume here that the cops minimize the number of rounds needed for capture, while the robber maximizes the number of rounds for capture.  For example, if $n\ge 1$ is an integer, then $\zeta(K_{1,n}) = 1$ and $\mathrm{capt}_{\zeta}(K_{1,n}) =n-1$.

There is an analogous temporal parameter defined for the cop number, also referred to as capture time, first introduced in \cite{bghk}. Since we focus almost exclusively on the localization number in this paper, there will be no confusion between the parameters. Note that as $k$ increases, $\mathrm{capt}_{\zeta ,k}(G)$ monotonically decreases. In analogy with \cite{over}, we refer to this in \emph{temporal speed-up} and the number of cops $k> \zeta(G)$ as \emph{overprescribed}. If $k=\beta(G),$ then $\mathrm{capt}_{\zeta,k}(G)=1.$

Capture time was implicitly defined in the first paper of Seager \cite{seager1}, where only one cop played. The parameter there was called the \emph{location number}, written $\mathrm{loc}(G),$ which is defined for $G$ with $\zeta(G)=1$ as the number of probes (that is, rounds) needed to capture the robber. The location number was studied in several subsequent publications, such as \cite{BDELM,car,has,seager2}. The present paper is the first place where the capture time is explicitly introduced in the general setting, where there may be more than one cop and possibly more cops than the localization number.

The capture time of a graph is challenging to calculate exactly for general graphs, and as such, we present bounds for general graph families, and a few exact values along the way. We present a conjecture on the capture time in Section~2 claiming that the capture time is linear in the order of the graph. We show the conjecture holds for graph families such as trees and interval graphs. We discuss the monotone property for the capture time on induced subgraphs. In particular, we show that capture time is monotone on induced subgraphs in a general family of graphs including trees. The capture time of trees is considered in Section~3, and we derive upper and lower bounds in Theorems~\ref{ttree} and \ref{ttree1}, respectively. We derive nearly tight values for the capture time on perfect $m$-ary trees in the overprescribed setting. In Section~4, we consider the capture time of the incidence graphs of projective planes of order $q.$ In \cite{BHM}, it was shown that such graphs have localization number $q+1.$ In Theorem~\ref{thm:captProj}, we show that if $G$ is the incidence graph of a projective plane of order $q$, then for $k \ge \zeta(G),$
\[\capt_{\zeta,k}(G) \leq \left\lceil \frac{q-1}{k-q}\right\rceil + \left\lceil \frac{q}{k-q+1}\right\rceil.\]
We present novel bounds on the localization number and capture time in terms of the treewidth of a graph. The final section includes several open problems.

Throughout, all graphs considered are simple, undirected, connected, and finite. For a general reference for graph theory, see~\cite{West}. The \emph{closed neighborhood} of $u$, written $N[u],$ consists of a vertex $u$ along with neighbors of $u$. The \emph{distance} between vertices $u$ and $v$ is denoted by $d(u,v).$ The maximum degree of a graph $G$ is denoted by $\Delta(G).$ For a graph $G,$ we denote the subgraph induced by $H$ by $G[H]$. A \emph{leaf} is a vertex of degree 1.

\section{Well-localizable graphs and the monotone property}

A basic question is how large the capture time may be as a function of the order of the graph. One motivation for this question comes from graph searching, where a set of searchers attempts to clear edges of contamination by prescribed rules; for more background on graph searching, see \cite{fomin} and the recent survey \cite{nisses}. In graph searching, it was shown in \cite{BiSe91,LaPa88} that for a graph of order $n,$ there exist winning strategies that require at most $n$ rounds. In particular, searchers never need to reclear edges. In contrast, the capture time for the cop number may be superlinear in the number of vertices. For graphs with cop number $k\ge 2$, it was proved in \cite{bi} that the capture time is $O(n^{k+1})$. Recently, the $O(n^{k+1})$ bound was proven to be asymptotically tight in \cite{br,WB}. Note that for graphs of order $n$ with cop number 1, the capture time is $O(n);$ see \cite{bghk}.

In the setting of the localization game, we do not know of graphs $G$ where the $\mathrm{capt}_{\zeta}(G)$ is superlinear in the order of $G$. It may be the case that the capture time is at most $n$, where $n$ is order of the graph, but we propose the following more modest conjecture.
\medskip

\noindent \emph{Localization Capture Time Conjecture (LCTC)}: If $G$ is a graph of order $n,$ then $$\mathrm{capt}_{\zeta}(G) = O(n).$$

\medskip

We say that a graph $G$ with order $n$ is \emph{well-localizable} if it satisfies $\mathrm{capt}_{\zeta}(G) = O(n).$ The LCTC may be rephrased as saying that all graphs are well-localizable.

One important family where LCTC holds is for trees.
\begin{theorem}\label{treew}
Trees are well-localizable.
\end{theorem}
\begin{proof} Let $T$ be a tree of order $n$. We then have that $\zeta(T) \le 2,$ by \cite{nisse1}. In the case $\zeta(T)=1,$ from the cop strategy presented in the proof of Theorem~8 in \cite{seager2}, the robber is captured in at most $n$ rounds, so the LCTC holds.

In the case when $\zeta(T) =2,$ consider the following strategy from \cite{nisse2}. Place one cop $C_1$ on a fixed vertex $u,$ and let the second cop $C_2$ probe each neighbor of $u,$ looking for a subtree $T_1$ of $T-u,$ where the distance to the robber is smallest. Once $T_1$ is discovered, we move
$C_1$ to the unique vertex of $T_1$ adjacent to $u$ and repeat this procedure on $T_1.$ We recursively search $T$ in this way, and in at most $n$ rounds, identify the robber's location. Hence, the LCTC holds in this case. \end{proof}

Another family where the LCTC holds is for \emph{interval graphs}, which are the intersection graphs of intervals on the real line. Interval graphs are precisely those that are chordal and asteroidal-triple-free; for more background on this graph family, see for example, \cite{spin}.

To prove that interval graphs are well-localizable, we first need the notions of treewidth and pathwidth. Given a graph $G$, a \textit{tree decomposition} is a pair $(X, T)$, where $X = \{B_1, B_2, \dots, B_m\}$ is a
family of subsets of $V(G)$ called \emph{bags}, and $T$ is a tree whose vertices are the subsets $B_i$, satisfying the following three properties.
\begin{enumerate}
\item $V(G)=\bigcup_{i=1}^mB_i.$ That is, each graph vertex is associated with
at least one tree vertex.

\item For every edge $(v, w)$ in the graph, there is a subset $B_i$ that
contains both $v$ and $w$.

\item If $B_i$, $B_j$ and $B_k$ are vertices, and $B_k$  is on the path from $B_i$ to $B_j,$ then $B_i\cap B_j \subseteq B_k$.
\end{enumerate}
The \textit{width} of a tree decomposition is the cardinality of its largest set $B_i$ minus one. The \textit{treewidth} of a graph $G,$ written $\mathrm{tw}(G),$ is
the minimum width among all possible tree decompositions of $G$. We refer to a tree decomposition with width equaling the treewidth as \emph{optimal}. If we restrict $T$ to be a path, then the resulting parameter is called the \emph{pathwidth}\index{pathwidth} of $G$, written $\mathrm{pw}(G).$ Note that $\mathrm{tw}(G) \le \mathrm{pw}(G).$

We have the following.
\begin{theorem}
If $G$ satisfies $\mathrm{pw}(G) = \zeta(G),$ then $G$ is well-localizable. In particular, interval graphs are well-localizable.
\end{theorem}
\begin{proof} As was shown in \cite{nisse1},  for all graphs $G$, we have that $\zeta(G) \le \mathrm{pw}(G).$ In the proof of this bound, the bags are linearly ordered as $B_i,$ where $1\le i \le m.$ The cops begin by occupying every vertex except one of $B_1,$
which ensures the robber will not start in $B_1$, or they are captured in one round. The cops then move to $B_2,$ occupying all but one vertex, and then $B_3,$ and so on. In this fashion, the robber can never enter a bag presently or one previously occupied by cops.
The robber will be eventually be captured in $B_m$ (given that they are omniscient and will maximize the length of the game).

The cop strategy described in the previous paragraph takes at most $n$ rounds to capture the robber. Hence, if $\mathrm{pw}(G) = \zeta(G),$ then $\mathrm{capt}_{\zeta}(G) \le n.$

Interval graphs satisfy $\mathrm{pw}(G) = \zeta(G),$ and so the final statement of the theorem holds. \end{proof}

We next prove that the LCTC holds for complete $k$-partite graphs. We let $\chi(G)$ denote the chromatic number of $G.$
\begin{theorem}\label{cor: complete kpartite, mod1}
For $k$ a positive integer, let $G$ be a complete $k$-partite graph of order $n,$ with parts $X_{i}$, for $1\leq i \leq k$,
 such that $|X_i| \leq |X_{i+1}|$ for $1 \leq i \leq k-1$.  Let $\rho$ be the number of parts of cardinality $1$. If $|X_{k}| > 1,$ then the following statements hold.
 \begin{enumerate}
 \item If $\rho \geq 1$, then we have that
$$\zeta(G) = n-\chi(G)-|X_{k}|+\rho+1 \text{ and } \capt_\zeta(G) \leq |X_k|-1;$$
\item If $\rho =0$, then we have that
$$\zeta(G) = n-\chi(G)-|X_{k}|+2 \text{ and } \capt_\zeta(G) \leq |X_k|-1.$$
\end{enumerate}
In particular, $\capt_{\zeta}(G) \le n-1.$
\end{theorem}
\begin{proof}
Let $m=n-\chi(G)-|X_{k}|+\rho+1$ if $\rho >0$, and $m= n-\chi(G)-|X_{k}|+2$ if $\rho=0$.
If there are fewer than $m$ cops, then this implies that either: (i) there is a part $X_i$ not containing the robber that will contain at least two cop-free vertices on the next cop turn; or (ii) there will be two parts of cardinality one that are cop-free on the next turn (although this case does not happen if $\rho \leq 1$).

In (i), the robber may move to one of the two vertices of $X_i$ that will be cop-free, and since the cops cannot distinguish these two vertices on the same part, the robber avoids capture.
In (ii), the robber may move to either of the two parts of cardinality one that are cop-free on the next turn, as the cops will not be able to distinguish these two.
Therefore, we have that $\zeta(G) \geq m$.

To show that $\zeta(G) \leq m$, we play with $m$ cops and show that the robber can be captured. We place $\rho-1$ cops on each independent set of cardinality $1$ if $\rho >0$.
There are  $n-\chi(G)-|X_{k}|+2$ cops left to place. We place $|X_i|-1$ cops on each non-singleton set $X_i$ where $i \neq k$. This consists of $n-k-|X_k|+1$ cops, so there is one cop left to place.
Each cop placed so far remains stationary for the rest of play. Finally, we place the last remaining cop on a vertex in $X_k$, which will move during play.

If the robber moves to one of the vertices that contain a cop on each turn, then the robber is captured. If the robber moves to the only part of cardinality $1$ that does not contain a cop, then all cops probe a distance of $1$, which results in a unique candidate and the robber is captured. If the robber moves to a cop-free vertex of a non-singleton part $X_i$, where $i \neq k$, then all cops on that part probe a distance of $2$, and all other cops probe a distance of $1$, which results in a unique candidate and so the robber is captured.

Therefore, the robber will be captured if they ever move from the part $X_k$.
As $X_k$ is an independent set, the robber cannot move between vertices of $X_k$.
On each successive turn, the cop player plays the unique moving cop on a new vertex of $X_k$.
After at most $|X_k|-1$ rounds, the cop player will either occupy the vertex of the robber, or the robber will be known to be on the unique vertex of $X_k$ that the cop has not yet visited.
\end{proof}

\subsection{The monotone property on induced subgraphs}
When studying capture time, it is useful to know how induced subgraphs affect it. We consider whether the capture time is \emph{monotone} on induced subgraphs: that is, if a graph $G$ has an induced subgraph $H$, then $\capt_{\zeta}(H)\leq \capt_{\zeta}(G)$. We show that capture time is monotone on induced subgraphs for a general family of graphs that includes all trees. The latter fact will be useful in the next section.

In general, capture time may fail to be monotone on induced subgraphs. We define $H$ as in Figure~\ref{treet}.
\begin{figure}[H]
     \centering
     \begin{tikzpicture}
\node[draw, circle, fill=black] (1) at (0,0) {};
\node[draw, circle, fill=black] (5) at (-2,-2) {};
\node[draw, circle, fill=black] (6) at (0,-2) {};
\node[draw, circle, fill=black] (7) at (2,-2) {};
\node[draw, circle, fill=black] (2) at (-1,-1) {};
\node[draw, circle, fill=black] (3) at (0,-1) {};
\node[draw, circle, fill=black] (4) at (1,-1) {};

\node[draw, circle, fill=black] (8) at (-2.5,-3) {};
\node[draw, circle, fill=black] (9) at (-1.5,-3) {};
\node[draw, circle, fill=black] (10) at (0.5,-3) {};
\node[draw, circle, fill=black] (11) at (-0.5,-3) {};
\node[draw, circle, fill=black] (12) at (1.5,-3) {};
\node[draw, circle, fill=black] (13) at (2.5,-3) {};

\draw [thick] (1)--(2)--(5)--(8)--(5)--(9);
\draw [thick] (1)--(3)--(6)--(10)--(6)--(11);
\draw [thick] (1)--(4)--(7)--(12)--(7)--(13);

\end{tikzpicture}
     \caption{Graph $H$.}
     \label{treet}
 \end{figure}
By \cite{seager2} we know that $\zeta(H)=1$. Furthermore, it is straightforward to check that $\capt_{\zeta}(H) = 3$. The graph $H$ is an induced subgraph of the Cartesian grid $G=P_8 \Box P_8$. Label the vertices of $G$ by $\{(i,j):1\le i,j \le 8 \}.$ Placing cops on $(1,1)$ and $(1,8)$ results in a capture in one round. Hence, $\zeta(G)=2,$ and $\mathrm{capt}_{\zeta}(G)=1.$

We consider certain subgraphs where the capture time is monotone on induced subgraphs in a strong sense, where it holds for all $k$ greater than the localization number.  We say that an induced subgraph $H$ of $G$ is \emph{special} if $H$ is connected and the only paths in $G$ between distinct vertices in $H$ are entirely contained in $H.$ Note that we assume here that paths do not repeat vertices or edges.
\begin{theorem} \label{lem:GnH}
Let $k \ge \zeta(G)$ be an integer. If $G$ is a connected graph and $H$ is a special induced subgraph of $G,$ then
$$\capt_{\zeta,k}(H) \leq \capt_{\zeta,k}(G).$$
\end{theorem}
Note that each $2$-connected component (or block) of $G$ is a special induced subgraph. Hence, by Theorem~\ref{lem:GnH}, $\capt_{\zeta,k}(G)$ is bounded below by the maximum capture time of any $2$-connected component.

Before we begin the proof of Theorem~\ref{lem:GnH}, we recall some notation. An induced subgraph $H$ is a \emph{retract} of $G$ if there is a homomorphism $f: G \rightarrow H$ such that $f(x)=x$ for $x\in V(H);$ that is, $f$ is the identity on $H.$ The map $f$ is called a \emph{retraction}.
\begin{proof}[Proof of Theorem~\ref{lem:GnH}.]
If we consider the graph formed by removing the edges of $H$ from $G$, then the graph is composed of a number of  components, say $X_1, X_2,\ldots, X_\alpha$, each of which contains exactly one vertex in $V(H)$. We note that any $X_i$ cannot contain two vertices in $V(H)$, or else a path exists between the two vertices that is not entirely contained within $H$. Further, there cannot be a $X_i$ with no vertices in $V(H)$, or else $G$ must not have been connected. We let $v_i$ denote the unique vertex of $V(H)$ that is contained in $X_i$. See Figure~\ref{fig:GnH}.

\begin{figure}[h]
\centering
\begin{tikzpicture}
 \draw (0,0) circle (50pt);
 \draw (70pt,0) circle (20pt);
 \draw (-90pt,0) circle (40pt);
 \draw (0,70pt) circle (20pt);
 \draw (0,-70pt) circle (20pt);
 \draw (0,70pt) node {$X_1$};
 \draw (0,-70pt) node {$X_2$};
 \draw (70pt,0) node {$X_3$};
 \draw (-90pt,0) node {$X_4$};
 \fill (50pt,0) circle (3pt);
 \fill (-50pt,0) circle (3pt);
 \fill (0,50pt) circle (3pt);
 \fill (0,-50pt) circle (3pt);
 \draw (-40pt,-5pt) node {$v_4$};
 \draw (40pt,0) node {$v_3$};
 \draw (0,-40pt) node {$v_2$};
 \draw (0,40pt) node {$v_1$};
 \draw (0,0) node {$H$};
 \draw (-80pt, 10pt) node {$u$};
 \fill (-70pt,10pt) circle (3pt);
 \draw[snake=zigzag] (-70pt,10pt) -- (-50pt,0pt);
 \draw (-20pt, 10pt) node {$v$};
 \fill (-30pt,10pt) circle (3pt);
 \draw[snake=zigzag] (-30pt,10pt) -- (-50pt,0pt);
\end{tikzpicture}
\caption{A special induced subgraph of $G$.}
\label{fig:GnH}
\end{figure}
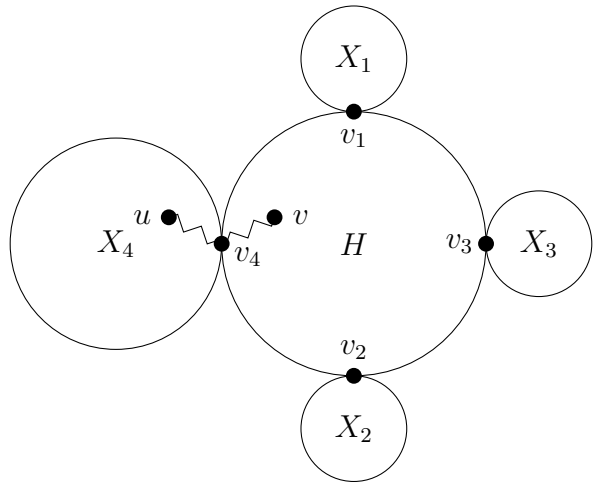
For each $u \in X_i$ and $v \in V(H)$, any path connecting $u$ to $v$ must contain $v_i$. Further, a minimal path between $u$ and $v$ must be composed of a minimal path from $u$ to $v_i$ and a minimal path from $v_i$ to $v$.
Hence,
$$
d(u,v) = d(u,v_i) + d(v_i,v).
$$
For $d\geq d(u,v_i)$, those vertices in $N^d(u) \cap V(H)$ are exactly those vertices in $N^{d-d(u,v_i)}(v_i) \cap V(H)$, where $N^d(u)$ denotes the set of vertices of distance $d$ from $u$.
Define the function $f:V(G) \rightarrow V(H)$ by $f(u) = v_i$ if $u \in X_i$ and by $f(u)=u$ if $u \in V(H)$. Note that this mapping is a retraction $f$ onto $H$ if we consider $G$ to be reflexive (that is, each vertex has a loop). Making the assumption that $G$ is reflexive has no impact on distances in the graph, and so we assume it without loss of generality for the remainder of the proof.

Suppose the robber restricts to playing in $H$, the cop knows the robber is in $H,$ and the cops may probe any vertex in $G.$ In this variant of the localization game played on $G,$ let $\capt_{\zeta,k}(G,H)$ be the resulting capture time. Note that $\capt_{\zeta,k}(G,H) \le \capt_{\zeta,k}(G).$
We show that $\capt_{\zeta,k}(H) \leq \capt_{\zeta,k}(G,H)$, which will be enough to complete the proof.

We employ a \emph{shadow strategy}, where the primary game is in $G$, and the secondary game is in $H$. The images of a cop $C \in V(G)$ onto $H$ are referred to as \emph{shadow cops} and written $f(C).$ Suppose that in round $t-1$, after the cop has moved, the cops knew the robber was on the vertex set $V_{t-1}$.
At time $t$, the cops play on vertices $C_1, C_2,\ldots, C_m$ in the primary game and as a result, the cop receives a distance vector $(d_1, d_2,\ldots, d_m)$.
The cop $C_i$ finds that the robber is on $N^{d_i}(C_i) \cap V(H)$ as a result of its probe, since it is known that the robber is on $V(H)$.
Therefore, the cop exactly knows that the robber is on the vertex set $$V_{t} = \big(N[V_{t-1}] \cap V(H)\big) \cap \bigcap_{i=1}^m \big(N^{d_i}(C_i) \cap V(H)\big),$$ where $N[V_{t-1}]$ is any vertex in $V_{t-1}$ or a neighbor of a vertex in $V_{t-1}$.
We previously showed in $G$ that $N^d(u) \cap V(H)=N^{d-d(u,f(u))}(f(u)) \cap V(H)$, so therefore, we may rewrite this as $$V_{t} = \big(N[V_{t-1}] \cap V(H)\big) \cap \bigcap_{i=1}^m \big(N^{d_i-d(C_i,f(C_i))}(f(C_i)) \cap V(H)\big).$$

Suppose for the sake of contradiction that on the primary game, the cop can capture the robber in $\capt_{\zeta,k}(H)-1$ rounds.
Suppose in round $t$, the cops probe vertices $C_1, C_2,\ldots, C_m$.
The robber translates this to the secondary game by playing shadow cops on vertices in $H$, yielding a distance vector $(d'_1, d'_2,\ldots, d'_m)$.
In round $t$, the shadow cops know that the robber is on the set $V_t'$.
The shadow cop $f(C_i)$ finds that the robber is on $N^{d'_i}(f(C_i)) = N^{d'_i}(f(C_i)) \cap V(H)$ as a result of its probe.
Therefore, the shadow cop exactly knows that the robber is on the vertex set $$V'_{t} = \big(N[V'_{t-1}]\big) \cap \bigcap_{i=1}^m  \big(N^{d'_i}(f(C_i)) \big)= \big(N[V'_{t-1}] \cap V(H)\big) \cap \bigcap_{i=1}^m \big(N^{d'_i}(f(C_i)) \cap V(H)\big).$$
By noting that $d_i' = d_i - d(C_i,f(C_i))$, we have that $V_i = V'_i$.
Consequentially, as $|V_t|=1$ only when $|V'_t|=1$, we have that the robber is captured in the secondary game in $\capt_{\zeta,k}(H)-1$.
This is a contradiction, as we require at least $\capt_{\zeta,k}(H)$ rounds to capture the robber with $k$ cops on $H$. Therefore, in the primary game, the cop requires at least  $\capt_{\zeta,k}(H)$ rounds to capture the robber. We then have that $\capt_{\zeta,k}(H) \leq \capt_{\zeta,k}(G,H)$, as required.
\end{proof}

The connected induced subgraphs of trees are always special, and so Theorem~\ref{lem:GnH} gives the following.
\begin{corollary}\label{tcor}
Capture time is monotone on trees. In particular, if $G$ is a tree with subtree $H,$ then for all integers $k \ge \zeta(G)$, $\capt_{\zeta,k}(H) \le \capt_{\zeta,k}(G).$
\end{corollary}

Note that for $k$ larger than the localization number, we may find graphs where the capture time when playing with $k$ cops is 1, but is unbounded for induced subgraphs. Consider the hypercube of dimension $n\ge 0$, written $Q_n$. As proved in \cite{BK}, $\loc{Q_n} \le \ceil{\log_2 n} + 2.$ If we label the vertices of $Q_n$ by binary $n$-tuples, we can place a cop on each of the vertices whose label contains at most one $1$ to see that $\capt_{\zeta, n+1}(G)=1$. However, $Q_n$ is bipartite and so has an independent set $X$ of order $2^{n-1}$. If $H=G[X]$, then $\capt_{\zeta, n+1}(H)= \lceil \frac{2^{n-1}-1}{n+1} \rceil.$

\section{Trees}
Trees have localization number either 1 or 2, as proved in \cite{seager2}. As we proved in Theorem~\ref{treet}, trees are well-localizable. We refine bounds on the capture time of trees in the overprescribed setting (that is, where the number of cops is more than the localization number), presenting bounds that are functions of the maximum degree.

When studying the capture time of trees, leaves play a special role.
\begin{lemma} \label{lem:all_leaves_win}
For an integer $j\geq 2$, if $T$ is a tree with $j$ leaves, then we have that $\mathrm{capt}_{\zeta,j}(T) =1$.
\end{lemma}
\begin{proof}
The cops play on the $j$ leaves. Each vertex on a tree is contained in a shortest path connecting two leaves. This implies that the robber, if not on a leaf and captured immediately, is on the shortest path connecting two cops, say $C_1$ and $C_2$. Suppose that cop $C_i$ has distance $d_i$ to the robber, for $i=1,2$. These two cops can identify that the robber is on the shortest path connecting $C_1$ and $C_2$ as $d_1+d_2 = d(C_1,C_2)$. The robber is identified to be on the unique vertex on the shortest path connecting $C_1$ and $C_2$ that is distance $d_1$ from $C_1$.
\end{proof}

For the proof of the next lemma, and other upcoming results, we need the following definition. Let $T$ be a rooted tree with root $v$. We call a vertex $a$ a \emph{descendant} of vertex $b$ in $T$ if $a$ is a neighbor of $b$, and $b$ sits on the unique path between $a$ and the root $v$.

When there are two or more cops we may assume without loss of generality that they play exclusively on the leaves.
\begin{lemma} \label{lem:cop2leaf}
If there are at least two cops playing the localization game on a tree $T$, then it never hinders the cops to be on leaves.
\end{lemma}
\begin{proof}
Suppose that in some round, there is a cop $C_1$ probing a vertex $u_1$ that is not a leaf of $T$. Let the other cops probe vertices $u_2, u_3, \ldots, u_k,$ respectively.
Say that the robber is on vertex $R$.

Root the tree at $u_2$, and let $u_1' \ne u_1$ be a leaf that is a descendant of $u_1$. Suppose $C_1$ is instead placed on $u_1'$ (and all other cops remain where they were). Let $P'$ be the shortest path between $u_1'$ and $u_2$ and let $P$ be the shortest path between $u_1$ and $u_2$. Observe that $P \subseteq P'$. Between them, cops $C_1$ and $C_2$ can determine the vertex $x$ on $P'$ that is closest to the robber. To see this, note that $$d(R,x) = \left(d(u_1',R) + d(u_2,R) - d(u_1',u_2)\right)/2.$$
From this, it follows that
$$d(u_1',x) = d(u_1',R) - d(R,x) =  \left(d(u_1',R) - d(u_2,R) + d(u_1',u_2)\right)/2$$ and this, together with knowing it lies on $P'$, is enough to determine $x$.

Supposing $C_1$ probes $u_1'$ and not $u_1$, we show that the cops know $d(u_1,R)$.
Whether $x \in P$ or $x \in P' \setminus P$, we have $d(u_1,R) = d(u_1,x) + d(x,R)$. Since the cops know $x$ and $d(x,R)$, the cops can still determine $d(u_1,R)$. Thus, playing cop $C_1$ at leaf $u_1'$ is  at least as good as playing the cop at $u_1$.
We may repeat this argument until all cops are on leaves.
\end{proof}

Throughout the remainder of the section, we let $m$ and $h$ be positive integers. The \emph{perfect $m$-ary} tree $\T_m^h$ of height $h$ has $h+1$ levels labeled $0,1,\ldots, h$. The $0$th level contains one vertex, the root, which has $m$ neighbors. For $0 < i < h$, every vertex in the $i$th level has exactly one neighbor in level $i-1$ and $m$ neighbors in level $i+1$. In particular, the $i$th level contains $m^i$ vertices all at distance $i$ from the root. These are naturally arising structures; for example, there is a natural bijection between the vertices of $\T_m^h$ and the set of  sequences of length at most $h$ with entries in $\{1,2,\ldots,m\}$, where two vertices are adjacent if one of the corresponding sequences is equal to the other sequence with one entry appended at the end.

Note that the perfect $m$-ary tree $\T_m^h$ has $m^h$ leaves. Also, for any $0 \le i \le h$ the induced subgraph on all vertices at distance $i$ or more from the root consists of $m^i$ copies of $\T_m^{h-i}$.

We will prove upper and lower bounds on the capture time for perfect $m$-ary trees.  Our analysis splits into two cases, depending on whether the number of cops is less than or greater than $m$.

We begin with a result that is effective when the number of cops $k$ is less than $m$.
\begin{lemma} \label{lem:large_m_lower_bound}
For $k\ge 2$ an integer, we have that
\[\capt_{\zeta,k}(\T_m^h ) \ge h \floor{ \frac{m-1}{k}} .\]
\end{lemma}
\begin{proof}
We will show that the robber can avoid capture if less than  $h \floor{ (m-1)/k}$ rounds are played. We assume that the robber picks a leaf and remains on that leaf for the remaining rounds, and the cops know that the robber is employing this strategy. The robber has other options but if the cops cannot capture the robber in fewer rounds in this case, they will not be able to do so when the robber has more freedom.

Let $T_1, T_2,\ldots, T_{m}$ denote the $m$ rooted subtrees formed by deleting the root vertex $u$, with the root of each such subtree being the vertex that was adjacent to $u$.
Each cop probes at most one of these subtrees in the first round.
During the first $\floor{ (m-1)/k}$ rounds, at most $m-1$ of these subtrees was visited by a cop.
The robber may therefore have chosen to start on a leaf of one of the subtrees, say $T_m$, that was not visited by a cop in these first $\floor{ (m-1)/k}$ rounds.
There may have been multiple subtrees that were not visited by the cops, so suppose that the robber tells the cop that they are on $T_m$ after these  $\floor{ (m-1)/k}$ rounds.
 This only serves to reduce the capture time.

Note that the distance from a cop on vertex $v$ to any of the leaf vertices of an unplayed subtree is $d(u,v) + h$, so the cops cannot distinguish the leaves of an unplayed subtree.
We may now proceed inductively, since the unplayed subtree is a copy of $\mathcal{T}_{m}^{h-1}$.
We repeat the argument over $h$ iterations, each taking $\floor{ (m-1)/k}$ rounds, yielding a capture time of at least $h\floor{ \frac{(m-1)}{k}}$.
\end{proof}

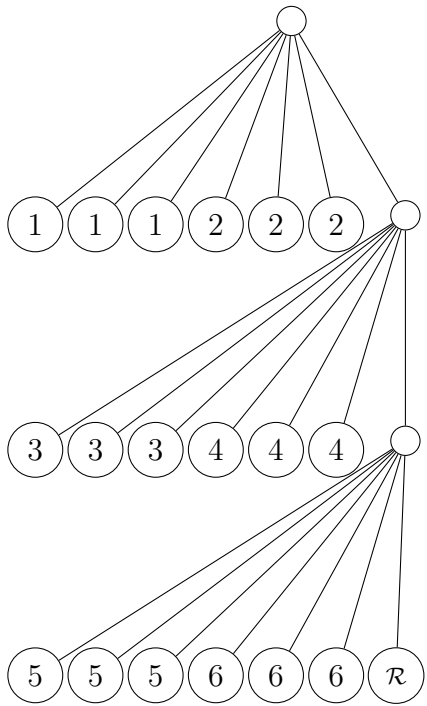
\begin{figure}[h]
\centering
\begin{tikzpicture}
\node [circle,draw] (1){};
\node [circle,draw] [below left = 2.3cm and 3cm of 1] (2){$1$};
\node [circle,draw] [below left = 2.3cm and 2.2cm of 1] (3){$1$};
\node [circle,draw] [below left = 2.3cm and 1.4cm of 1] (4){$1$};
\node [circle,draw] [below left = 2.3cm and 0.6cm of 1] (5){$2$};
\node [circle,draw] [below left = 2.3cm and -0.2cm of 1] (6){$2$};
\node [circle,draw] [below left = 2.3cm and -1cm of 1] (7){$2$};
\node [circle,draw] [below left = 2.3cm and -1.8cm of 1] (8){};
\node [circle,draw] [below left = 5.3cm and 3cm of 1] (2b){$3$};
\node [circle,draw] [below left = 5.3cm and 2.2cm of 1] (3b){$3$};
\node [circle,draw] [below left = 5.3cm and 1.4cm of 1] (4b){$3$};
\node [circle,draw] [below left = 5.3cm and 0.6cm of 1] (5b){$4$};
\node [circle,draw] [below left = 5.3cm and -0.2cm of 1] (6b){$4$};
\node [circle,draw] [below left = 5.3cm and -1cm of 1] (7b){$4$};
\node [circle,draw] [below left = 5.3cm and -1.8cm of 1] (8b){};
\node [circle,draw] [below left = 8.3cm and 3cm of 1] (2c){$5$};
\node [circle,draw] [below left = 8.3cm and 2.2cm of 1] (3c){$5$};
\node [circle,draw] [below left = 8.3cm and 1.4cm of 1] (4c){$5$};
\node [circle,draw] [below left = 8.3cm and 0.6cm of 1] (5c){$6$};
\node [circle,draw] [below left = 8.3cm and -0.2cm of 1] (6c){$6$};
\node [circle,draw] [below left = 8.3cm and -1cm of 1] (7c){$6$};
\node [circle,draw] [below left = 8.3cm and -1.8cm of 1] (8c){\fontsize{9.2}{1}\selectfont $\mathcal{R}$};
\path[draw] (1) -- (2);
\path[draw] (1) -- (3);
\path[draw] (1) -- (4);
\path[draw] (1) -- (5);
\path[draw] (1) -- (6);
\path[draw] (1) -- (7);
\path[draw] (1) -- (8);
\path[draw] (8) -- (2b);
\path[draw] (8) -- (3b);
\path[draw] (8) -- (4b);
\path[draw] (8) -- (5b);
\path[draw] (8) -- (6b);
\path[draw] (8) -- (7b);
\path[draw] (8) -- (8b);
\path[draw] (8b) -- (2c);
\path[draw] (8b) -- (3c);
\path[draw] (8b) -- (4c);
\path[draw] (8b) -- (5c);
\path[draw] (8b) -- (6c);
\path[draw] (8b) -- (7c);
\path[draw] (8b) -- (8c);

\end{tikzpicture}
\caption{An example of three cops playing on $\mathcal{T}_{7}^{3}$ as in the proof of Lemma~\ref{lem:large_m_lower_bound}. Each vertex containing a number $t$ represents a subtree that the cops visited in round $t$.
}
\label{fig:large_m}
\end{figure}

See Figure~\ref{fig:large_m} for an example of Lemma~\ref{lem:large_m_lower_bound}. We derive an upper bound that differs from the lower bound by an additive factor of at most $h$.
\begin{lemma}  \label{lem:large_m_upper_bound}
For an integer $k \ge 2$ we have that
$$\capt_{\zeta,k}(\T_m^h ) \le h \ceil{  \frac{m-1}{k} } .$$
\end{lemma}
\begin{proof}
We modify the proof of Lemma~\ref{lem:large_m_lower_bound} to instead show an upper bound. Let $T_1, \ldots, T_{m}$ denote the $m$ rooted subtrees formed by deleting the root vertex $u$, with the root $u_i$ of each such subtree $T_i$ being the vertex that was connected to $u$.

In round $t$ with $1 \le t \le \ceil{ (m-1)/k },$ the cops play on $u_{k(t-1)+1}, u_{k(t-1)+2}, \ldots, u_{kt} $. Note that if the robber ever plays on the root $u$ they are captured immediately, as the root $u$ is the unique vertex at distance $1$ from all the cops (this requires at least two cops). Thus, the robber must remain on the same $T_i$ for all these rounds. We may assume without loss of generality that if $i = k(t-1) + j$ for some $1 \le t \le \ceil{ (m-1)/k }$ and $1 \le j \le k$, then in round $t$ the cops can determine that the robber was in $T_i$, and that the cops could not know this before round $t$. This follows since $d(u_{j'}, R) =d(u_{j''}, R)$ for all $j', j'' \ne i$ and  $d(u_i, R) =d(u_{j'}, R) – 2$ for all $j' \ne i$, where $R$ is the vertex occupied by the robber that lies on the subtree rooted at $u_i$. Hence, $u_i$ will be the unique vertex with a shorter distance to the robber than the others. If $i$ is not of this form, then $i = m$ and after $t$ rounds we deduce the robber is in $T_m$ as it was not in any $T_i$.

After at most $\ceil{ (m-1)/k }$ rounds the cops know which subtree the robber is on and the robber can never move to the root of the tree without being captured. Now we move to playing on $T_i$ and apply induction. Since the robber can never move to the root $u_i$, the robber can never escape the subtree $T_i$. This induction lasts for $h$ rounds.
\end{proof}

We therefore have the following.
\begin{corollary}
For all integers $k\ge 2,$ we have that
\[
h \floor{  \frac{m-1}{k} } \le \capt_{\zeta,k}(\mathcal{T}_m^h) \le h \ceil{  \frac{m-1}{k} }.
\]
\end{corollary}

We note that this corollary is most useful when the number of cops $k$ is less than $m$. If $k \ge m$, then these bounds only show $0 \le \capt_{\zeta,k}(\mathcal{T}_{m}^{h}) \le h$ and we can find much improved bounds.

We next present a result that is effective in the case when the number of cops is greater than or equal $m$.
\begin{lemma}\label{lem:small_m_lower_bound}
For an integer $k\ge 2$, we have that
$$\capt_{\zeta,k}( \T_m^h) \ge \frac{h}{1 + \floor{\log_m{k}}} .$$
\end{lemma}
\begin{proof}
Since we are only proving a lower bound, we may  assume that the robber picks a leaf and remains on that leaf throughout the game and the cops know that the robber is employing this strategy; the robber has other options but if the cops cannot capture the robber in fewer rounds in this case, they will not be able to do so when the robber has more freedom.

Let $i = 1 + \floor{\log_m{k}}$ so that $m^{i-1} \le k < m^i$.
We claim that in round $t$ there is some subtree $T$ isomorphic to $\T_m^{h - ti}$ such that the set of vertices where the robber could be is exactly the leaves of $T$. In round $0$, the initial robber placement could be on any leaf of $\T_m^h$ so clearly this holds in round $0$. We proceed by induction on $t$.

Suppose that the claim holds in round $t$ for the subtree $T$ isomorphic to $\T_m^{h - ti}$. Consider the induced subgraph on all vertices of $T$ at distance $i$ or more from the root of $T$.
We know this consists of $m^i$ copies of  $\T_m^{h - ti - i}$. The cops pick $k < m^i$ vertices, so there is one of these subtrees isomorphic to  $\T_m^{h - ti - i}$ with no cops on it. Call this $T'$ and suppose that the robber chose a leaf of this subtree; this is possible since the robber knew in advance which vertices the cops would pick. For any vertex $v$ outside $T',$ the distance from $v$ to a leaf of $T'$ is the same for every leaf of $T'$, so the cops cannot distinguish between the leaves of $T'$. In particular, the set of vertices the robber could be on contains all leaves of $T'$, and the claim holds in round $t+1$.

Now, it is evident that the cops cannot capture the robber in $t$ rounds for $t < h/i$. Hence, $h-it > 0$, and the tree $\T_m^{h-it}$ has at least $ m$ leaves, and so there are at least $ m > 1$ vertices that the robber could be on in round $t$.
\end{proof}

\begin{figure}[h]
\centering
\begin{tikzpicture}
\node [circle,draw] (1){};
\node [circle,draw] [below left = 1.5cm and 1.5cm of 1] (2){};
\node [circle,draw] [below right = 1.5cm and 1.5cm of 1] (3){};
\node [circle,draw] [below left = 1.5cm and 0.7cm of 2] (4){$1$};
\node [circle,draw] [below right = 1.5cm and 0.5cm of 2] (5){$1$};
\node [circle,draw] [below left = 1.5cm and 0.5cm of 3] (6){$1$};
\node [circle,draw] [below right = 1.5cm and 0.7cm of 3] (7){$*$};
\node [circle,draw] [below left = 1.5cm and 0.2cm of 4] (8){};
\node [circle,draw] [below right = 1.5cm and 0.2cm of 4] (9){};
\node [circle,draw] [below left = 1.5cm and 0.2cm of 5] (10){};
\node [circle,draw] [below right = 1.5cm and 0.2cm of 5] (11){};
\node [circle,draw] [below left = 1.5cm and 0.2cm of 6] (12){};
\node [circle,draw] [below right = 1.5cm and 0.2cm of 6] (13){};
\node [circle,draw] [below left = 1.5cm and 0.2cm of 7] (14){\tiny $*$};
\node [circle,draw] [below right = 1.5cm and 0.2cm of 7] (15){\tiny $*$};
\path[draw] (1) -- (2);
\path[draw] (1) -- (3);
\path[draw] (2) -- (4);
\path[draw] (2) -- (5);
\path[draw] (3) -- (6);
\path[draw] (3) -- (7);
\path[draw] (4) -- (8);
\path[draw] (4) -- (9);
\path[draw] (5) -- (10);
\path[draw] (5) -- (11);
\path[draw] (6) -- (12);
\path[draw] (6) -- (13);
\path[draw] (7) -- (14);
\path[draw] (7) -- (15);
\end{tikzpicture}
\caption{An example for the proof of Lemma~\ref{lem:small_m_lower_bound}, where $3$ cops are playing on $\mathcal{T}_{2}^{h}$, which forces the robber to be on the subtree indicated by stars.}
\label{fig:small_m}
\end{figure}
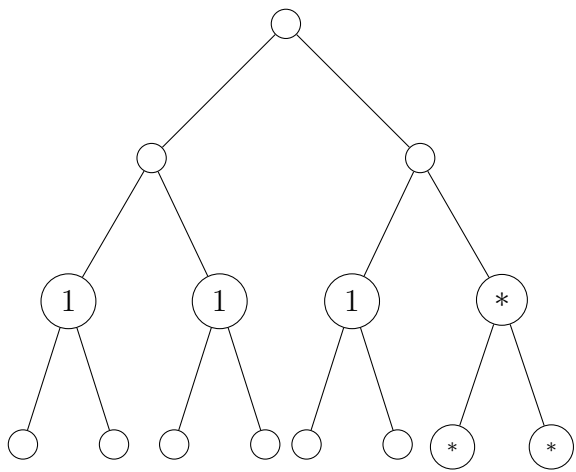

We can also find an analogous upper bound.

\begin{lemma}\label{lem:small_m_upper_bound}
If $m \le k < m^{h+1}$, then we have that
$$\capt_{\zeta,k}( \T_m^h ) \le \frac{h}{\floor{\log_m{k}}} .$$
\end{lemma}
\begin{proof}
Let $i = \floor{\log_m{k}}$ and note that $1 \le i \le h$. It is enough to  show that the theorem holds with $k = m^i$ cops owing to temporal speed-up.

We will prove the lemma by induction on $h$. When $h \le i$, by Lemma~\ref{lem:all_leaves_win} the cops can win in one round by playing on all of the leaves of $\T_m^h$, and so we are done.

Suppose $h > i$. In the first round, the cops play on the $m^i$ vertices at distance $i$ from the root. Call this set $C$. If the robber was on a vertex $v$ at distance $j$ from the root, then the maximum distance probed from $C$ is exactly $i + j$, so the cops can determine $j$.

If $j \le i$, then the  $2^{i-j}$ vertices in $C$ that are descendants of $v$ are at distance $i-j$ from $v$ and all other leaves are at distance $i+j$ from $v$. Thus, given the set of distances, we can determine all leaves that are descendants of the vertex the robber is on and taking their unique common ancestor gives $v$.

Otherwise, $j > i$ and there is a unique vertex $u$ in $C$ with the minimum distance $j-i$ to $v$. The cops know that the robber is on the subtree rooted at $u$ so on future rounds play passes to this subtree. This subtree is isomorphic to $\T_m^{h-i}$ and so by induction the cops can find the robber in a further $(h-i)/i = h/i - 1$ rounds. Thus, the cops can capture the robber in $h/i$ rounds.
\end{proof}

We therefore have the following.
\begin{corollary} \label{cor:large_k}
For $k \geq m,$
\[
\frac{h}{ 1+ \floor{\log_m(k)}} \leq \capt_{\zeta,k}(\mathcal{T}_m^h) \leq  \frac{h}{\floor{ \log_m(k)}}
\]
\end{corollary}

The upper and lower bounds differ by a multiplicative factor of $ 1 + \frac{1}{\floor{\log_m{k}}}$, so if $k$ is large compared to $m$ this difference is relatively small.

Every tree $T$ of maximum degree $\Delta(T)$ and radius $h=\mathrm{rad}(T)$ is a subtree of $\mathcal{T}_{\Delta}^h$. By Corollary~\ref{tcor} and Lemmas~\ref{lem:small_m_upper_bound} and \ref{lem:large_m_upper_bound}, we derive the following.

\begin{theorem}\label{ttree}
For a tree $T$, we have that:
\begin{equation*}
\capt_{\zeta,k}(T) \leq \
\begin{cases} \mathrm{rad}(T) \ceil{ \frac{\Delta(T)-1}{k} } &\text{if } 2 \le k < \Delta(T), \\
\frac{\mathrm{rad}(T)}{\floor{ \log_{\Delta(T)}(k) }} &\text{if } k \ge \Delta(T).\\
\end{cases}
\end{equation*}
\end{theorem}

As for a lower bound, we know that a tree of maximum degree $\Delta(T)$ contains a star $K_{1,\Delta(T)}$, so by the monotone property of capture time on trees we have the following.
\begin{theorem}\label{ttree1}
For a tree $T$, if $k < \Delta(T),$ then we have that
\begin{equation*}
\capt_{\zeta,k}(T) \ge
 \ceil{ \frac{\Delta(T)-1}{k} }.
\end{equation*}
\end{theorem}

\section{Incidence Graphs of Projective Planes}

A \emph{projective plane} consists of a set of points and lines satisfying the following axioms.

\begin{enumerate}
\item There is exactly one line incident with every pair of distinct points.

\item There is exactly one point incident with every pair of distinct lines.

\item There are four points such that no line is incident with more than two
of them.
\end{enumerate}

We only consider finite projective planes. It can be shown that projective planes have $q^2+q+1$ many points and lines for an integer $q$, each point is on $(q+1)$-many lines, and each line contains $(q+1)$-many points. We let $X$ be the set of points and $\mathcal{B}$ be the set of lines. The only orders where projective planes are known to exist are prime powers; indeed, this is a deep conjecture in finite geometry. For these items and further background on projective planes, see \cite{cam}.

For a given projective plane $P$, define  the bipartite graph with red vertices the points of $P$, and the blue vertices represent the lines. Vertices of different colors are adjacent if they are
incident. We call this the \emph{incidence graph of} $P$. See Figure~\ref{1fanogp} for the incidence graph of the Fano plane, the projective
plane of order $2$.
\begin{figure} [ht!]
\begin{center}
\epsfig{figure=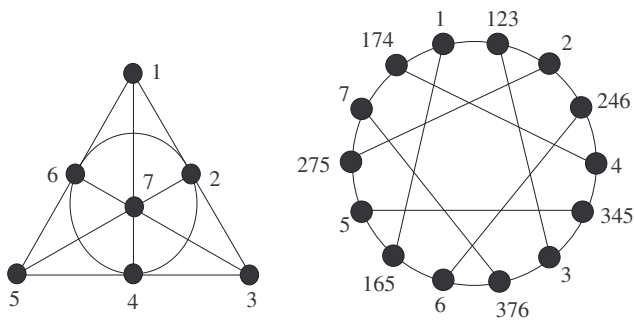,scale=0.7}\caption{The Fano plane and its incidence graph, the \emph{Heawood graph}. Lines are represented by triples.}\label{1fanogp}
\end{center}
\end{figure}

As was proved in \cite{BHM}, the incidence graph $G$ of a projective plane $P$ of order $q$ satisfies $\zeta(G)=q+1.$ We present an upper bound on the capture time of these graphs in the present section.

The following is a warm-up for larger values of $q.$
\begin{lemma}
The capture time of the Heawood graph is $2$.
\end{lemma}
\begin{proof}
We identify the Heawood graph as the incidence graph of the Fano plane with $X = \{1,2,3,4,5,6,7\}$ and $\mathcal{B} = \{123, 174, 165, 246, 275, 345, 376\}$. For the first round, the cops probe vertices 1, 4 and 6 on $X.$ Each vertex in $\mathcal{B}$ is uniquely identifiable so the robber must reside on $X.$ Next, move the cops to 2, 3, and 5. All of the vertices of $\mathcal{B}$ are uniquely identifiable so the robber could not have moved over there. If the robber was on 1,2,3,4,5 or 6, then they were immediately identified when the vertex was probed. Otherwise, the robber is on 7, the unique vertex that was not ever visited by cops and is of distance 2 from all the cops. Observe that the first round reduced the possible location of the robber in $X$ to exclude $\{1,4,6\}$ while the second step reduced $X$ to exclude $\{2,3,5\}$ leaving only $7$ as a possibility to probe distance $2$ without previously being identified.

We show that the robber cannot be captured in one round. We identify the ways in which the cops can be distributed among the vertices. There are four cases to consider: (1) all three are on $X$, (2) all three are on $\mathcal{B}$, (3) two are on $X$ and one is on $\mathcal{B}$, or (4) one is on $X$ and two are on $\mathcal{B}$. As the design is symmetric, checking cases (1) and (3) suffice. We leave the remaining details to the reader.
\end{proof}

The following bounds the capture time of the incidence graphs of projective planes of order.
\begin{theorem} \label{thm:captProj}
Let $G$ be the incidence graph of a projective plane of order $q\ge 2$.
For $k\ge q+1$ an integer, we have that
\[\capt_{\zeta,k}(G) \leq \left\lceil \frac{q-1}{k-q}\right\rceil + \left\lceil \frac{q}{k-q+1}\right\rceil.\]
\end{theorem}

In the case $k=q+1,$ Theorem~\ref{thm:captProj} gives that for $G$ the incidence graph of a projective plane of order $q\ge 2,$ $$\capt_{\zeta}(G) \le q-1 + \left\lceil \frac{q}{2}\right\rceil.$$ Hence, the incidence graphs of projective planes are well-localizable and so satisfy the LCTC.
\begin{proof}
As we wish to show an upper bound on the capture time, we simply need a cop strategy where $k$ cops can capture the robber in $\left\lceil \frac{q-1}{k-q}\right\rceil + \left\lceil \frac{q}{k-q+1}\right\rceil$ rounds.
We do this over two phases. In the first phase, the cops take $t_1$ rounds to find that the robber is on a set of vertices $N(u)$ for any vertex $u$ in $G$.
In the second phase, the cops take $t_2 = \text{capt}_{\zeta,k}(G) - t_1$ additional rounds to capture the robber.
We bound both $t_1$ and $t_2$.

We start by providing the strategy for the first phase. Let $u_1\in \mathcal{B}$ and $u_2\in N(u_1)$.
The cops play on the $q$ vertices in $A=N(u_2) \setminus \{u_1\}$, and without loss of generality, play on $k-q \leq q-1$ of the vertices in $N(u_1) \setminus \{u_2\}$, labeled as $B_1$. Note that we may make this assumption since capture time monotonically decreases with $k$, and for larger $k$ the ceilings give the same upper bound.

If the robber is on $X \setminus \{u_2\}$, then there is a unique path of length two from $u_2$ to the robber, say $u_2xr$.
If there is a cop on $x$, then this cop will probe a distance of $1$, so  the cops know that the robber is on $N(x)$, and the first phase would end.
If there is not a cop on $x$, then $x = u_1$, and each cop in $A$ will probe a distance of $3$, which can only happen if the robber is on a neighbor of $u_1$, and the first phase would end.
The cops on $A$ will all probe a distance of $1$ if and only if the robber is on $u_2$, so the first phase ends if the robber is on $X$.

We may therefore assume the robber is on a vertex in $\mathcal{B}$.
If the first phase does not end on the first round, then the cops know that the robber is not on $A$ or on a neighbor of a vertex in $B_1$.
This means that the robber must reside on the vertex $u_1$ or on a neighbor of a vertex in $N(u_1) \setminus B_1$.
The cops on $B$ each probe $1$ in the case the robber is on $u_1$, so the first phase would immediately be over if this were the case.

On the $n$th round, the cops play on $A$ and $B_n$, where $B_n$ is a set of $k-q$ vertices chosen from  $N(u_1) \setminus (\{u_2\} \cup \bigcup_{i=1}^n B_i)$.
By similar reasoning to the above, the first phase ends unless the robber is on a neighbor of $N(u_1) \setminus (\{u_2\} \cup \bigcup_{i=1}^n  B_i)$.
After $\left\lceil \frac{q-1}{k-q}\right\rceil$ rounds, either some cop probes a distance of $1$ and so the first phase is over, or there is one vertex in $N(u_1) \setminus (\{u_2\} \cup \bigcup_{i=1}^n  B_i)$, and the robber is known by the cops to be residing on the neighborhood of this vertex.
Therefore, $t_1 \leq  \left\lceil \frac{q-1}{k-q}\right\rceil$.

We now give the strategy for the second phase.
Suppose that the robber was found to reside on $C\subseteq N(u)$, for some vertex $u$, where $|C| = \alpha \leq q+1$.

We first suppose that $\alpha \geq 3$.
Let $v$ be a vertex on the same part as $u$.
The cops play $\alpha-1$ vertices on $C$, and $k-\alpha+1$ vertices on $N(v)$, which we label as $D$.
If some cop on $C$ probes a distance of $0$, then the robber is captured.
If all cops on $C$ probe a distance of $2$, then the robber is on the unique vertex of $C$ that does not contain a cop.
Therefore, the robber cannot stay on a vertex of $C$ without being captured, so we assume the robber moves.
Suppose the robber moved from vertex $x \in C$ to a neighboring vertex $N(x) \setminus \{u\}$. If the robber moved to $u$, then they would be immediately captured by the cops.
The cop knows that the robber is on $N(x) \setminus \{u\}$.
Each cop on $D$ has exactly one neighbor on $N(x)$, and no two cops on $D$ share the same neighbor on $N(x)\setminus \{u\}$.
The robber is captured if they are on the neighbor of a cop on $D$.
We may therefore assume that the robber is on one of the $q-k+\alpha-1<\alpha$ other vertices of $N(x) \setminus \{u\}$, which implies that the cops  have found that the robber is residing on $C'\subseteq N(u)$, for some vertex $u$, where $|C'| = \alpha' = \alpha - (k-q+1)<\alpha$.

Now if $\alpha=2$, the cops instead play $2$ vertices on $C$, and $k-2$ vertices on $N(v)$, which we label as $D$.
If some cop on $C$ probes a distance of $0$, then the robber if found. Otherwise the robber has moved during its last turn.
If both cops on $C$ probe a distance of $1$, then the robber is on $u$.
Otherwise, one of the cops on $C$, say on vertex $x$, probes a distance of $1$, and the cops know the robber is on $N(x) \setminus\{u\}$.
Each of the $k-2 \geq q-1$ cops in $B$ is adjacent to one vertex of $N(x) \setminus\{u\}$, and vice versa.
Thus, if one of the cops in $D$ probes a distance of $1$, the cops know the exact location of the robber.
If all of the cops in $D$ probe a distance of $3$, then the robber must be on the unique vertex of $N(x) \setminus\{u\}$ without a neighbor in $D$.

We repeat the argument above inductively for $\lceil \frac{q}{k-q+1}\rceil$ rounds, in which case there can be at most one vertex that the robber could be residing on, and the robber is captured.\end{proof}

\section{Bounds using treewidth}\label{sec: bipartite and k-partite graphs}
We consider new bounds on the localization number and capture time using treewidth. Define the \emph{tree radius} of a graph $G$, written $\mathrm{tr}(G)$, as the minimum radius of all tree decompositions of $G$ that have width $\mathrm{tw}(G)$. We have the following result.
\begin{theorem}
For a graph $G$, we have that
\[ \zeta(G) \leq (\mathrm{tw}(G)+1) \mathrm{tr}(G).
\]
If an optimal tree decomposition has $L$ leaves, then we have that
$$\mathrm{capt}_{\zeta ,(\mathrm{tw}(G)+1)\mathrm{tr}(G)} \le L.$$
\end{theorem}
\begin{proof}
Let $X_i$ be the bags of a tree decomposition of $G$ with width $\mathrm{tw}(G)$ and radius $\mathrm{tr}(G)$.
Label each bag as \emph{unvisited}.
We begin by selecting a vertex $X_1$ that is in the center of the tree decomposition.
For the unique path from $X_1$ to some leaf of distance $\mathrm{tr}(G)$
from $X_1$, say $(X_1 =Y_1, Y_2,\ldots, Y_r)$, place a cop on each vertex of $G$ in the vertex set $Y_1 \cup Y_2 \cup \ldots \cup Y_r$.
Label each bag in $(Y_1, Y_2,\ldots, Y_r)$ as \emph{occupied}.
This set of vertices contains at most $(\mathrm{tw}(G)+1) \mathrm{tr}(G)$ vertices.
If the robber is not on the set of vertices $Y_1 \cup Y_2 \cup \dots \cup Y_r$, then the robber is not captured and so moves.

The cop now chooses a new path that intersects the last path in as many vertices as possible.
For example, if there is a leaf $U$ in the tree decomposition that is also adjacent to $Y_{r-1}$, then the cop may choose the path $(X_1 =Y_1,Y_2, \ldots ,Y_{r-1},U)$.
We label the bags that are no longer in the current path as \emph{visited}.
Note that by the cop strategy,
\begin{enumerate}
\item no visited bag will be revisited by the cops;
\item no unvisited bag is adjacent to a visited bag in the tree decomposition; and
\item if a bag $X$ is unvisited in the current round, then each of the neighbors of $X$ in the next round will be either unvisited or occupied.
\end{enumerate}
In particular, note that if some vertex of the original graph $v$ is not contained in an occupied bag, then $v$ may be contained in occupied bags or unoccupied bags, but not both.

We may assume that the robber was previously on a vertex $v$ with the property that all bags that contain $v$ are unvisited.
If the robber moved to a vertex $w$ that is now contained in an occupied bag, then the robber is captured as each such vertex contains a cop.
Supposing then that $w$ is not in an occupied bag, then by (3) and since $v$ was only in unvisited bags, $w$ can only be in unvisited bags, unless the robber is captured.
The cop continues this process until all paths starting at $X_1$ and ending at a leaf (of distance at most $\mathrm{tr}(G)$ from $X_1$) have been probed.
Eventually there will be no unvisited bags, by which time the robber has been captured.
As such, the cop must have one round for each leaf, and so the number of rounds needed for capture with this strategy is at most $L.$
\end{proof}

We also have the following upper bounds.
 \begin{theorem}
For a graph $G$ we have that
\[ \zeta(G) \leq (\mathrm{tw}(G)+1) (\Delta(G)+1),
\]
and
$$\mathrm{capt}_{\zeta,(\mathrm{tw}(G)+1) (\Delta(G)+1)} \le \mathrm{tr}(G)+1.$$
\end{theorem}
\begin{proof}
Let $X_i$ be bags of a tree decomposition of $G$ with width $\mathrm{tw}(G)$ and radius $\mathrm{tr}(G)$.
We begin by selecting a vertex $X_1$ that is in the center of the tree decomposition.
Using at most $(\mathrm{tw}(G)+1) (\Delta(G)+1)$ cops, we place a cop on each vertex of $X_1$, and a cop on each neighbor of a vertex in $X_1$.
If we suppose that the robber is not on the same vertex as a cop, then the cop with the smallest distance to the robber must lie on a neighboring vertex of $X_1$, and not in $X_1$ itself.

Suppose that a closest cop to the robber was on vertex $u_1$.
Note that $u_1$ is not in $X_1$.
If we delete the bag $X_1$ from the tree decomposition, then a forest is formed where there is exactly one subtree with a bag containing $u_1$.
We can then take $X_2$ to be the unique bag that is in this subtree and which was connected to $X_1$ in the tree decomposition.
This procedure can then be recursively applied, where on each round a new bag $X_i$ is selected that is further away from $X_1$ than the bag $X_{i-1}$ in the previous round.
We note that by this procedure, if the robber is on vertex $R$ during turn $i$ and has avoided capture, then every bag $B$ that contains $R$ has the property that the unique path from $B$ to $X_1$ in the tree decomposition contains the path $(X_i, X_{i-1}, \dots, X_2, X_1)$.
We repeat the procedure until a bag is chosen that is a leaf of the tree decomposition, by which time the robber must have been captured. This winning strategy for the cops takes at most $\mathrm{tr}(G)+1$ rounds.
\end{proof}

\section{Further directions}

The capture time for the localization game was introduced and studied for several graph families. We introduced the Localization Capture Time Conjecture (LCTC) and proved it for trees, interval graphs, and complete $k$-partite graphs. The monotone property of capture time on induced subgraphs was investigated, and we gave bounds on the capture time of trees. Upper bounds on the localization number and capture time were found in terms of treewidth. We presented upper bounds on the capture time of the incidence graphs of projective planes, verifying that those graphs also satisfy the LCTC.

Apart from the LCTC, there are other open questions surrounding the capture time. We noticed that in many cases, if $G$ is a graph of order $n$ and $k \ge 2$, then $\capt_{\zeta, \zeta+k}(G) = O(n/k)$. It would be interesting to see if this strengthening of the LCTC holds in general. There is more work to be done on upper bounds for the capture time on trees, and the bound on the capture time of incidence graphs of projective planes need tightening. One final question we present is whether computing the capture time on graphs is \textbf{NP}-hard.

\end{document}